\documentclass[british,5p, sort&compress, final]{elsarticle}
\usepackage[T1]{fontenc}
\usepackage[latin9]{inputenc}
\usepackage{amsmath}
\usepackage{amsthm}
\usepackage{amssymb}

\makeatletter

\newcommand{\noun}[1]{\textsc{#1}}

\numberwithin{equation}{section}
\numberwithin{figure}{section}
\theoremstyle{plain}
\newtheorem{thm}{\protect\theoremname}

\@ifundefined{date}{}{\date{}}
\usepackage{nag}
\usepackage{bm}
\usepackage[final]{microtype}
\usepackage{xcolor}
\usepackage{setspace}
\usepackage{float}
\usepackage[unicode=true,bookmarks=true,breaklinks=true,bookmarksnumbered=false,bookmarksopen=false,pdfborder={0 0 0},pdfborderstyle={},backref=false,colorlinks=true,allcolors=cyan]{hyperref}
\usepackage[english,ruled, linesnumbered]{algorithm2e}
\usepackage{doi}
\usepackage[titletoc]{appendix}
\usepackage{mathtools, cuted}
\allowdisplaybreaks

\makeatother

\usepackage{babel}
\providecommand{\theoremname}{Theorem}

\begin{document}

\begin{frontmatter}{}

\title{How to count the number of zeros that a polynomial has on the unit
circle?}

\author{R. S. Vieira}

\address{\noun{e-mail}\emph{:} ricardovieira@ufscar.br \hspace{1cm}\emph{\noun{
orcid}}\emph{: 0000-0002-8343-7106.}\\
}

\address{Universidade Federal de São Carlos (UFSCar), Departamento de Matemática,
\\
Rod. Washington Luís, Km. 235, SP-310, CEP. 13565-905, São Carlos,
São Paulo, Brasil. \\
}
\begin{abstract}
The classical problem of counting the number of real zeros of a real
polynomial was solved a long time ago by Sturm. The analogous problem
of counting the number of zeros that a polynomial has on the unit
circle is, however, still an open problem. In this paper, we show
that the second problem can be reduced to the first one through the
use of a suitable pair of Möbius transformations --- often called
Cayley transformations --- that have the property of mapping the
unit circle onto the real line and vice versa. Although the method
applies to arbitrary complex polynomials, we discuss in detail several
classes of polynomials with symmetric zeros as, for instance, the
cases of self-conjugate, self-adjoint, self-inversive, self-reciprocal
and skew-reciprocal polynomials. Finally, an application of this method
to Salem polynomials and to polynomials with small Mahler measure
is also discussed.
\end{abstract}

\end{frontmatter}{}

\section{\noun{\label{Section-Sturm}}Methods for counting the number of real
zeros of real polynomials}

The first exact method for counting the number of zeros that a real
polynomial has on the real line (or in a given interval of the real
line) was presented by Sturm in 1829 (see \citep{JeanClaudePadovani2009},
p. 323). In its simplest form, Sturm algorithm works as follows: given
a real polynomial $p(z)$ of degree $n$, let $a<b$ be two real numbers
which are not a multiple zero of $p(z)$. Then, construct the so-called
\emph{Sturm sequence}\footnote{This is the \emph{classical Sturm sequence}. For the definition of
more general Sturm sequences, see \citep{Akritas1989}.}, $S(z)=\left\{ S_{0}(z),S_{1}(z),S_{2}(z),\ldots,S_{m}(z)\right\} ,$
whose elements are defined as follows: 
\begin{multline}
S_{0}(z)=p(z),\qquad S_{1}(z)=p'(z),\qquad\text{and}\\
S_{k}(z)=-\text{rem}\left[S_{k-2}(z),S_{k-1}(z)\right],\qquad2\leqslant k\leqslant m,
\end{multline}
 where $p^{\prime}(z)$ is the derivative of $p(z)$, $\mathrm{rem}\left[A,B\right]$
denotes the remainder of the polynomial division of $A$ by $B$ and
$m$ is the integer determined from the condition that $S_{m}(z)$
has degree zero. Now, let $\text{var}\left[S(\zeta)\right]$ denote
the number of sign variations in the sequence $S(z)$ for $z=\zeta$.
Then, Sturm showed that the number $N$ of distinct zeros of $p(z)$
in the half-open interval $\mathcal{I}=(a,b]$ is, just, $N=\text{var}\left[S(a)\right]-\text{var}\left[S(b)\right].$
This is the content of the so-called \emph{Sturm theorem}, whose proof
can be found in many places --- see, for example, \citep{Akritas1989}.
It works because as we vary $z$ from $a$ to $b$ on the real line
the sequence $S(z)$ suffers a sign variation when, and only when,
$z$ passes through a zero of $p(z)$; thus, the number of sign variations
of $S(z)$ from $a$ to $b$ exactly counts the number of distinct
real zeros of $p(z)$ in this interval. Considering the interval $\mathcal{I}$
as the whole real line, the number of real zeros of $p(z)$ is obtained
(Sturm algorithm in this case can be further simplified, as it is
enough to keep only the leading terms of the polynomials $S_{k}(z)$,
which asymptotically dominates the behaviour of these polynomials).

Notice that Sturm sequence is constructed in very similar fashion
as the sequence of remainders obtained in computation of the greatest
common divisor \noun{(gcd)} of the polynomials $p(z)$ and $p^{\prime}(z)$:
the only difference is that we should keep the \emph{opposite} of
the polynomial remainders in each step. Alternatively, we could compute
the ordinary sequence $R(z)=\left\{ r_{0}(z),r_{1}(z),r_{2}(z),\ldots,r_{m}(z)\right\} $
of the remainders obtained in the computation of the \noun{gcd}  of
$p(z)$ and $p^{\prime}(z)$ (where $r_{0}(z)=p(z)$ , $r_{1}(z)=p^{\prime}(z)$
and $r_{k}(z)=\text{rem}\left[r_{k-2}(z),r_{k-1}(z)\right]$, $2\leqslant k\leqslant m$),
from which the Sturm sequence can be obtained by negating the signs
of each two consecutive remainders $r_{k}(z)$ as follows: 
\begin{multline}
S(z)=\left\{ r_{0}(z),r_{1}(z),-r_{2}(z),-r_{3}(z),r_{4}(z),r_{5}(z),\right.\\
\left.-r_{6}(z),-r_{7}(z),r_{8}(z),r_{9}(z),\ldots\right\} .
\end{multline}
This can be easily shown by comparing the construction of the two
sequences and noticing that the quotients $t_{k}(z)$ appearing in
Sturm's sequence are related with the respective quotients $q_{k}(z)$,
obtained in the \noun{gcd} of $p(z)$ and $p^{\prime}(z)$, through
the formula $t_{k}(z)=(-1)^{k+1}q_{k}(z)$. This relationship shows
us that the complexity of Sturm algorithm is the same as the complexity
of the \noun{gcd} of $p(z)$ and $p^{\prime}(z)$.

We remark that Sturm's method requires $p(z)$ a real polynomial with
no multiple zeros at the endpoints $a$ and $b$, although $p(z)$
may have multiple zeros in the open interval $\left(a,b\right)$.
Notice also that the counting excludes that eventual zero at $z=a$
but includes the zero at $z=b$; it is, however, an easy matter to
verify if $p(z)$ has or not a zero at $z=a$, so that we can also
count the number of zeros of $p(z)$ in any closed interval $[a,b]$
of the real line. Besides, keep in mind that the Sturm algorithm counts
only the number of \emph{distinct} real zeros of $p(z)$. This issue,
however, can be overcome by additional analysis\footnote{Indeed, Sturm himself had shown in a subsequent paper (see \citep{JeanClaudePadovani2009},
p. 345) that the number of non-real zeros of $p(z)$ in the interval
$(a,b]$ can also be determined from his method by other arguments.
Moreover, from a generalization of Sturm algorithm due to Thomas \citep{Thomas1941},
the multiplicity of the zeros can be counted directly.}.

It is worth to mention that Sturm derived this theorem during his
researches on qualitative aspects of differential equations, which
gave rise to the so-called \emph{Sturm-Liouville theory}. In fact,
in an interval of weeks, Sturm published similar theorems regarding
the distribution of zeros of orthogonal functions, which are solutions
of Sturm-Liouville differential equation \citep{JeanClaudePadovani2009}.
Sturm was influenced by the works of Fourier and, as a matter of a
fact, his method can be thought of as a refinement of Fourier's previous
result \citep{Fourier1820} that establishes an upper bound for the
number of real zeros of $p(z)$ in a given half-open interval $(a,b]$
of the real line through the number of sign variations in the \emph{Fourier
sequence} $F(z)=\left\{ p(z),p'(z),\ldots,p^{(n)}(z)\right\} $, for
$z$ running from $a$ to $b$ over the real line. Thus, we can say
that Sturm's method makes Fourier's exact.

Since the publication of Sturm's fundamental papers, other methods
for counting or isolating the real zeros of a given real polynomial
were formulated. In 1834, Vincent published a paper \citep{Vincent1834}
(republished two years later, with few additions, in \citep{Vincent1836}),
in which a method based on successive replacements in terms of continued
fractions was proposed. His method was based on a previous work of
Budan \citep{Budan1807}, who established a theorem equivalent to
that of Fourier commented above, although in a different form. Unfortunately,
Vincent's work was almost forgotten thenceforward and, in fact, it
was only rescued from oblivion in 1976 by Collins and Akritas, who
formulated a powerful bisection method based on Vincent's theorem
for isolating the zeros of a given real polynomial \citep{CollinsAkritas1976}.
Two years later, Akritas \citep{Akritas1978} gave a fundamental contribution
to this method by replacing the uniform substitutions that take place
in Vincent's algorithm by non-uniform ones based on previously calculated
bounds for the zeros of the testing polynomial (with that modification,
Akritas was able to reduce the complexity of Vincent's method from
exponential to polynomial type). Further improvements of these methods,
among with new symbolic and numeric techniques, gave rise to some
of the fastest algorithms known to date for counting or isolating
the zeros of real polynomials on the real line \citep{AlesinaGaluzzi2000,RouillierZimmermann2004,AkritasStrzebonskiVigklas2008,SagraloffMehlhorn2016,KobelEtal2016,Vieira2017}
and also on regions of the complex plane \citep{Ehrlich1967,Aberth1973,Wilf1978,Mitsui1983,Bini1996,BrunettoEtal2000,Pan2002,Eeisermann2012}.

\section{The Cayley transformations and polynomials}

The methods described above determine the exact number of zeros of
a real polynomial on the real line $\mathbb{R}$. The correspondent
problem of determining the exact number of zeros of a given polynomial
on the unit circle is still unsolved. In fact, this is an old question
whose first works remount to the end of \noun{xix} century, for instance,
the pioneer works Eneström, Kakeya, Schur, Kempner, Cohn, among others
--- see \citep{Vieira2019} and references therein. In the recent
years, a great interest in this problem has emerged, usually in connection
with the theory of the so-called \emph{self-inversive} \emph{polynomials}.
These are complex polynomials whose zeros are all symmetric with respect
to the unit circle (real self-inversive polynomials includes the \emph{self-reciprocal}
and the \emph{skew-reciprocal} polynomials). These classes of polynomials
are very important in both pure and applied mathematics, as they appear
in several problems related to the theory of numbers, algebraic curves,
knots theory, stability theory, dynamic systems, error-correcting
codes, cryptography and even in classical, quantum and statistical
mechanics --- see \citep{Vieira2019} and references therein. An
important question regarding self-inversive and self-reciprocal polynomials
is the number of zeros that these polynomials have in the unit circle
$\mathbb{S}=\left\{ z\in\mathbb{C}:|z|=1\right\} $. There are a countless
number of papers devoted to the question of finding conditions for
all, some, or no zero of a self-inversive polynomial to lie on $\mathbb{S}$,
see \citep{Vieira2019}. 

In this paper, we present a method that reduces the problem of counting
the number of zeros that an arbitrary complex polynomial has on the
unit circle to the problem of counting the number of zeros of a real
polynomial on the real line. Because the second problem is completely
addressed by Sturm (or any other real root-counting) algorithm, our
approach also solves the first problem completely. The method is based
on the use of the following pair of Möbius transformations:
\begin{multline}
\mu(z)=\left(z-i\right)/\left(z+i\right),\qquad\text{and}\\
\omega(z)=-i\left(z+1\right)/\left(z-1\right),\label{Cayley}
\end{multline}
which are often called \emph{Cayley transformations}. Together with
the relations $\mu(\infty)=1$, $\mu(-i)=\infty$ and $\omega(1)=\infty$,
$\omega(\infty)=-i$, these two transformations become the inverse
of each other in the extended complex plane $\mathbb{C}_{\infty}=\mathbb{C}\cup\left\{ \infty\right\} $.
It can be easily verified that $\mu(z)$ maps the real line onto the
complex unit circle, while $\omega(z)$ maps the unit circle onto
the real line\footnote{We remark that the transformations (\ref{Cayley}) are not the only
pair of Möbius transformations that maps $\mathbb{S}$ onto $\mathbb{R}$
and vice versa: they are, however, the most adequate ones for our
purposes.}. Besides, $\mu(z)$ sends any point in the upper-half (lower-half)
plane to the interior (exterior) of $\mathbb{S}$, so that $\omega(z)$
sends any point in the inside (outside) of $\mathbb{S}$ to the upper-half
(lower-half) plane. 

Given a complex polynomial $p(z)$ of degree $n$, we define the \emph{transformed
polynomials} $q_{\mu}(z)$ and $q_{\omega}(z)$ by the formulas: 
\begin{multline}
q_{\mu}(z)=\left(z+i\right)^{n}p(\mu(z)),\qquad\text{and}\\
q_{\omega}(z)=\left(\tfrac{i}{2}\right)^{n}\left(z-1\right)^{n}p(\omega(z)).\label{q_mu}
\end{multline}
The factor $\left(\tfrac{i}{2}\right)^{n}$ in front of the second
formula in (\ref{q_mu}) is to make the two mappings the inverse each
of the other. The following theorems discuss some properties of these
transformed polynomials and their zeros.
\begin{thm}
Let $p(z)$ be a complex polynomial of degree $n$. If $p(z)$ has
a zero of multiplicity $m$ at the point $z=1$, then $q_{\mu}(z)$
defined as above will be a polynomial of degree $n-m$. Similarly,
if $p(z)$ has a zero of multiplicity $m$ at the point $z=-i$, then
$q_{\omega}(z)$ defined as above will be a polynomial of degree $n-m$. 
\end{thm}
\begin{proof}
Suppose that $p(z)$ has a zero at $z=1$ of multiplicity $m$, where
$0\leqslant m\leqslant n$. Write, $p(z)=\left(z-1\right)^{m}r(z)$,
where $r(z)$ is is a polynomial of degree $n-m$ with no zeros at
$z=1$. From (\ref{q_mu}) we get that $q_{\mu}(z)=\left(-2i\right)^{m}s(z)$,
where $s(z)=\left(z+i\right)^{n-m}r(\mu(z))$. Now, expanding $s(z)$
in powers of $z$ we can verify that its leading coefficient equals
$r(1)$; because $r(1)\neq0$ we conclude that $s(z)$ is a polynomial
of degree $n-m$ and so it is $q_{\mu}(z)$. By the same argument,
if $p(z)$ has a zero of multiplicity $m$ at the point $z=-i$, then,
from (\ref{q_mu}) we get that $q_{\omega}(z)$ will be is a polynomial
of degree $n-m$.
\end{proof}
Thus, the condition for the transformed polynomial $q_{\mu}(z)$ (respectively,
$q_{\omega}(z)$) to have the same degree as the original polynomial
$p(z)$ is that $p(z)$ has no zero at $z=1$ (respectively, at $z=-i$). 
\begin{thm}
\label{TheoremRoots}Let $\zeta_{1},\ldots,\zeta_{n}$ be the zeros
of a complex polynomial $p(z)$ of degree $n$. If $p(1)\neq0$, then
the zeros of the transformed polynomials $q_{\mu}(z)$ will be, respectively,
$\xi_{1}=\omega\left(\zeta_{1}\right),\ldots,\xi_{n}=\omega\left(\zeta_{n}\right)$.
Similarly, if $p(-i)\neq0$, then the zeros of the transformed polynomial
$q_{\omega}(z)$ will be, respectively, $\eta_{1}=\mu\left(\zeta_{1}\right),\ldots,\eta_{n}=\mu\left(\zeta_{n}\right)$.
\end{thm}
\begin{proof}
Inverting the first equation in (\ref{q_mu}), we get that $p\left(\zeta_{k}\right)=\left(\tfrac{i}{2}\right)^{n}\left(\zeta_{k}-1\right)^{n}q_{\mu}\left(\omega\left(\zeta_{k}\right)\right)=0$,
$1\leqslant k\leqslant n$, but $\zeta_{k}\neq1$ which means that
$\xi_{k}=\omega\left(\zeta_{k}\right)$ is a zero of $q_{\mu}(z)$.
Similarly, inverting the second equation in (\ref{q_mu}) we get that
$p\left(\zeta_{k}\right)=\left(\zeta_{k}+i\right)^{n}q_{\omega}\left(\mu\left(\zeta_{k}\right)\right)=0$,
$1\leqslant k\leqslant n$, and the condition $\zeta_{k}\neq1$ implies
that $\eta_{k}=\mu\left(\zeta_{k}\right)$ is a zero of $q_{\omega}(z)$.
\end{proof}
Theorem \ref{TheoremRoots} shows us that whenever a polynomial is
transformed through a Cayley transformation, its zeros are accordingly
transformed through the inverse transformation. Besides, from the
relations $\mu(-i)=\infty$ and $\omega(1)=\infty$, we see that if
$p(z)$ has a zero at the point $z=1$ (respectively, $z=-i$), then
the transformed polynomial $q_{\mu}(z)$ (respectively, $q_{\omega}(z)$)
will have a zero at infinity, which confirms again that the transformed
polynomial cannot have the same degree as $p(z)$ in these cases. 

The previous results imply the following theorem, which is a keystone
in what follows:
\begin{thm}
\label{TheoremNumberRoots}Let $p(z)$ be a complex polynomial of
degree $n$ that has $m$ zeros on $\mathbb{S}$, counted with multiplicity,
and such that $p(1)\neq0$. Then the transformed polynomial $q_{\mu}(z)$
will have exactly $m$ zeros on $\mathbb{R}$, also counted with multiplicity.
Similarly, if $p(z)$ is a complex polynomial of degree $n$ that
has $m$ zeros on $\mathbb{R}$, counted with multiplicity, and such
that $p(-i)\neq0$, then the transformed polynomial $q_{\omega}(z)$
will have $m$ zeros on $\mathbb{S}$, also counted with multiplicity.
\end{thm}
\begin{proof}
These statements follow directly from theorems proved above and from
the fact that the Cayley transformations $\mu(z)$ and $\omega(z)$
map $\mathbb{R}$ on $\mathbb{S}$ and vice versa, respectively. 
\end{proof}
The following complements Theorem \ref{TheoremNumberRoots}: 
\begin{thm}
\label{TheoremNumberSymmetric}Let $p(z)$ be a complex polynomial
of degree $n$ that has exactly $2m$ zeros symmetric to $\mathbb{S}$
and such that $p(1)\neq0$. Then, the polynomial $q_{\mu}(z)$ will
have exactly $2m$ zeros symmetric to $\mathbb{R}$ (i.e., complex
conjugate zeros). Conversely, if $p(z)$ is a complex polynomial of
degree $n$ that has precisely $2m$ zeros symmetric to $\mathbb{R}$
and such that $p(-i)\neq0$, then the polynomial $q_{\omega}(z)$
will have precisely $2m$ zeros symmetric to $\mathbb{S}$. 
\end{thm}
\begin{proof}
Let $\zeta$ and $1/\zeta^{\star}$ be any pair of zeros of $p(z)$
that are symmetric to $\mathbb{S}$\footnote{The star means complex conjugation so that, if $p(z)=\sum_{k=0}^{n}p_{k}z^{k}$,
then $p(z^{\star})=\sum_{k=0}^{n}p_{k}\left(z^{\star}\right)^{k}$,
$p^{\star}(z)=\sum_{k=0}^{n}p_{k}^{\star}\left(z^{\star}\right)^{k}$
and $p^{\star}(z^{\star})=\sum_{k=0}^{n}p_{k}^{\star}z^{k}$.}. The corresponding zeros of $q_{\mu}(z)$ will be $\xi=\omega(\zeta)$
and $\chi=\omega\left(1/\zeta^{\star}\right)$. However, from (\ref{Cayley})
we can easily show that $\omega\left(1/\zeta^{\star}\right)=\omega^{\star}(\zeta)$,
so that $\chi=\xi^{\star}$. Similarly, if $\zeta$ and $\zeta^{\star}$
are any complex conjugate pair of zeros of $p(z)$, then it follows
that the corresponding zeros of $q_{\omega}(z)$ are $\eta=\mu(\zeta)$
and $\sigma=\mu\left(\zeta^{\star}\right)$. But from (\ref{Cayley})
we can show that $\mu\left(\zeta^{\star}\right)=1/\mu^{\star}\left(\zeta\right)$,
so that $\sigma=1/\eta^{\star}$.
\end{proof}

\section{General complex polynomials\label{Section-GeneralComplex}}

It is clear from Theorem \ref{TheoremNumberRoots} how we can count
the number of zeros that a polynomial $p(z)$ of degree $n$ has on
the unit circle: all we need to do is to compute the transformed polynomial
$q_{\mu}(z)=\left(z+i\right)^{n}p(\mu(z))$ and then use some \emph{real-root-counting}
(\noun{rrc}) method to count the number of zeros of $q_{\mu}(z)$
on the real line\footnote{We mention that the idea of using a Möbius transformation to verify
if a given polynomial has some zero on the unit circle is not new,
although this topic seems to not have been explored in detail before.
Indeed, as far as we know, such possibility was discussed only in
some old references due to Kempner \citep{Kempner1916,Kempner1935}
and, more recently, in an expository note due to Conrad \citep{Conrad2016}
(who credited F. Rodriguez-Villegas for this idea). We remark, however,
that Kempner considered a only real polynomial $p(z)$, while the
transformed polynomial $q(z)$ was defined through the formula, $q(z)=\left(z^{2}+1\right)p\left(\frac{z-i}{z+i}\right)p\left(\frac{z+i}{z-i}\right)$;
this essentially corresponds to the case discussed by us in Algorithm
\ref{AlgConjugate}. Conrad, on the other hand, made no mention to
Sturm algorithm or any other \noun{rcc} method.}. Some care should be taken, however, depending on whether method
we use to this end. For example, as mentioned in Section \ref{Section-Sturm},
Sturm and Akritas methods do not take into account the multiplicity
of the zeros of the testing polynomials; if we want to account the
multiplicities, then a suitable \noun{rrc} method should be employed
to this end --- for example, Thomas algorithm \citep{Thomas1941}.
Besides, we shall see that the action of the Cayley transformation
over a polynomial $p(z)$ usually results in a non-real polynomial
even when $p(z)$ is real\footnote{We shall see in Theorem \ref{Theorem-SI-SC} that the transformed
polynomial $q_{\mu}(z)=\left(z+i\right)^{n}p(\mu(z))$ will be a real
polynomial only if the original polynomial $p(z)$ is self-adjoint.}, whereas most of the \noun{rrc} methods need a real polynomial to
work with --- including Sturm or Akritas methods. Thus, whenever
$q_{\mu}(z)$ is not a real polynomial, an auxiliary real polynomial
$Q(z)$, which has the same number of zeros on $\mathbb{R}$ than
$q_{\mu}(z)$, should be found. 

\vspace{0.5cm}
\begin{algorithm}
\SetAlgoRefName{1A} 
\caption{ \textsc{General complex polynomials}} 
\label{AlgComplex1A}
\SetKwInOut{Input}{input} \SetKwInOut{Output}{output}
\Input{
A complex polynomial $p(z)$ of degree $n$.}
\Output{The number of distinct zeros of $p(z)$ on the unit circle.}
$n\coloneqq \mathrm{degree}(p(z))$\;
$q(z)\coloneqq\left(z+i\right)^{n} p\left(\frac{z-i}{z+i}\right)$\;
\lIf{$q(z) \neq q^{\star}(z^{\star})$}{$q(z) \leftarrow q(z)q^{\star}(z^{\star})$ \textbf{end}}
$N\coloneqq \textsc{rrc}\left[q(z),-\infty,\infty \right]$\;
\lIf{$p(1)=0$}{$N\leftarrow N+1$ \textbf{end}}
\Return{$N$}.
\end{algorithm}\vspace{0.5cm}

In what follows, we shall present algorithms\footnote{For the sake of simplicity, hereafter we shall consider that the \noun{rrc}
method employed in the algorithms has the same properties as those
of Sturm's method. The symbol $\textsc{rrc}\left[q(z),\alpha,\beta\right]$
will denote an \noun{rrc} procedure that gives the exact number of
\emph{distinct} zeros that a real polynomial $q(z)$ has on the interval
$(\alpha,\beta]$ of the real line.} that allow one to compute the number of zeros that a polynomial $p(z)$
of degree $n$ has on $\mathbb{S}$. First we shall present general
algorithms that work with an arbitrary complex-polynomial; then, specific
algorithms that take into account the symmetry of the zeros of the
testing polynomials regarding $\mathbb{S}$ or $\mathbb{R}$ will
be presented. The asymptotic complexities of these algorithms are
all the same, as the \noun{rrc} procedures used in the algorithms
are the most time-consuming part of them (for the complexity of Sturm
and Akritas algorithms, see \citep{Akritas1989}). Nonetheless, when
comparing polynomials of the same degree we shall see that the specific
algorithms are usually faster because they deliver a polynomial of
smaller degree to the \noun{rrc} procedure. 

\vspace{0.5cm}
\begin{algorithm}
\SetAlgoRefName{1B} 
\caption{\textsc{General complex polynomials (alternative)}} 
\label{AlgComplex1B}
\SetKwInOut{Input}{input} \SetKwInOut{Output}{output}
\Input{
A complex polynomial $p(z)$ of degree $n$.}
\Output{The number of  distinct zeros of $p(z)$ on the unit circle.}
$n\coloneqq \mathrm{degree}(p(z))$\;
$q(z)\coloneqq\left(z+i\right)^{n} p\left(\frac{z-i}{z+i}\right)$\;
\If{$q(z) \neq q^{\star}(z^{\star})$}{
$r(z)\coloneqq \frac{1}{2}\left[q(z)+q^{\star}(z^{\star})\right]$\;
$s(z)\coloneqq \frac{1}{2i}\left[q(z)-q^{\star}(z^{\star})\right]$\;
$q(z)\leftarrow \textsc{gcd}[r(z),s(z)]$ }
$N\coloneqq \textsc{rrc}\left[q(z),-\infty,\infty \right]$\;
\lIf{$p(1)=0$}{$N\leftarrow N+1$ \textbf{end}}
\Return{$N$}.
\end{algorithm}
\vspace{0.5cm}

Let us begin with the case where the testing polynomial $p(z)$ is
an arbitrary complex polynomial of degree $n$. From (\ref{q_mu}),
it follows that the transformed polynomial $q_{\mu}(z)$ will usually
have non-real coefficients. Thus, provided that the \noun{rrc} procedure
works only with a real polynomial, we need to find an auxiliary polynomial
$Q(z)$, with real coefficients, that has the same number of zeros
on $\mathbb{R}$ as does $q_{\mu}(z)$. We can overcome this issue
in two ways: The first way consists of multiplying the transformed
polynomial $q_{\mu}(z)$ by its complex conjugate, $q_{\mu}^{\star}(z^{\star})$,
so that a polynomial of degree $2n$ is obtained in place, namely,
$Q(z)=q_{\mu}(z)q_{\mu}^{\star}(z^{\star}).$ It is clear that the
zeros of $q_{\mu}^{\star}(z^{\star})$ are the complex conjugate of
the zeros of $q_{\mu}(z)$, from which it follows that the real polynomial
$Q(z)$ has the same number of real zeros than $q_{\mu}(z)$, counted
without multiplicity, as required. Now we can use the \noun{rrc} procedure
to count the number of real zeros of $Q(z)$, which, according to
Theorem \ref{TheoremNumberRoots}, will correspond to the number of
zeros that the original polynomial $p(z)$ has on $\mathbb{S}$, provided
$p(1)\neq0$ (if $p(1)=0$ then all we need to do is to add $1$ to
the final result). This is described in Algorithm \ref{AlgComplex1A}.
The second way consists of writing the transformed polynomial in the
form $q_{\mu}(z)=r(z)+is(z)$, where $r(z)=\frac{1}{2}\left[q_{\mu}(z)+q_{\mu}^{\star}(z^{\star})\right]$
and $s(z)=\frac{1}{2i}\left[q_{\mu}(z)-q_{\mu}^{\star}(z^{\star})\right]$,
so that $r(z)$ and $s(z)$ are both real polynomials. Then we can
compute the \noun{gcd} of $r(z)$ and $s(z)$ and define $Q(z)=\textsc{gcd}\left[r(z),s(z)\right].$
It follows that the polynomial $Q(z)$ has degree utmost $n$ and,
in particular, it has the same number of zeros on $\mathbb{R}$ as
does the polynomial $q_{\mu}(z)$. This is the content of the following: 
\begin{thm}
\label{TheoremZerosGCD}The zeros of the polynomial \sloppy $Q(z)=\textsc{gcd}\left[r(z),s(z)\right]$
are precisely the zeros of $q_{\mu}(z)$ whose complex conjugate is
also a zero of $q_{\mu}(z)$. Thus, the degree of $Q(z)$ equals the
number of zeros of $q_{\mu}(z)$ whose complex conjugate is also a
zero of it, counted with multiplicity. 
\end{thm}
\begin{proof}
We can decompose the complex polynomial $q_{\mu}(z)$ in a product
of two polynomials, say, $q_{\mu}(z)=t(z)u(z)$, where $t(z)$ consists
of a real polynomial that gathers all the zeros of $q_{\mu}(z)$ which
appear in complex conjugate pairs (real zeros included), while $u(z)$
is a non-real polynomial that gathers all the remaining zeros of $q_{\mu}(z)$.
Thus, it follows from the definition of the polynomials $r(z)$ and
$s(z)$ given above that the polynomial $t(z)$ divides both $r(z)$
and $s(z)$, while $u(z)$ cannot divide both of them. Therefore,
as the zeros of $Q(z)=\textsc{gcd}\left[r(z),s(z)\right]$ correspond
to the common zeros of $r(z)$ and $s(z)$, the first result follows.
Finally, the degree of $Q(z)$ follows from the Fundamental Theorem
of Algebra.
\end{proof}
\vspace{0.5cm}
\begin{algorithm}
\setcounter{algocf}{1}
\caption{\textsc{General complex polynomials: zeros in an arc of the unit circle}} 
\label{AlgComplexAB}
\SetKwInOut{Input}{input} \SetKwInOut{Output}{output}
\Input{ A complex polynomial $p(z)$ of degree $n$ and two real numbers $\alpha$ and $\beta$ such that $0\leqslant \alpha < 2\pi$ and $0 < \beta \leqslant 2\pi$.}
\Output{ The number of distinct zeros of $p(z)$ on the arc $\mathcal{J}=(\mathrm{e}^{i \alpha},\mathrm{e}^{i \beta}]$ of the unit circle.}
\lIf{$\alpha=0$}{$a=-\infty$ \textbf{else} $a \coloneqq -i\left(\frac{\mathrm{e}^{i \alpha}+1}{\mathrm{e}^{i \alpha}-1}\right)$ \textbf{end}}
\lIf{$\beta=2\pi$}{$b = \infty$ \textbf{else} $b \coloneqq -i\left(\frac{\mathrm{e}^{i \beta}+1}{\mathrm{e}^{i \beta}-1}\right)$ \textbf{end}}
$n\coloneqq \mathrm{degree}(p(z))$\;
$q(z) \coloneqq \left(z+i\right)^{n} p\left(\frac{z-i}{z+i}\right)$\;
\lIf{$q(z) \neq q^{\star}(z^{\star})$}{$q(z) \leftarrow q(z)q^{\star}(z^{\star})$ \textbf{end}}
\If{$\alpha>\beta$}{$N \coloneqq \textsc{rrc}[q(z),-\infty,b]+\textsc{rrc}[q(z),a,\infty]$\;
  \lIf{$p(1)=0$}{$N \leftarrow N+1$ \textbf{end}}
\Return{$N$.}}
$N \coloneqq \textsc{rrc}\left[q(z),a,b \right]$\;
  \lIf{$p(1)=0 \ \mathbf{and} \ b=2\pi$}{$N \leftarrow N+1$ \textbf{end}}
\Return{$N$}.
\end{algorithm}
\vspace{0.5cm}

Notice as well that each common zero of $r(z)$ and $s(z)$ is also
a zero of $q_{\mu}(z)=r(z)+is(z)$; the converse, however, is not
true: if $\xi$ is a zero of $q_{\mu}(z)$, then $\xi$ can be either
a common zero of $r(z)$ and $s(z)$ or we should have $r(\xi)/s(\xi)=-i$.
The first case happens whenever $\xi$ is a zero of $Q(z)$ --- in
particular when $\xi$ is real --- whence, the second case occurs
whenever $\xi$ is a zero of $q_{\mu}(z)$ whose complex conjugate
$\xi{}^{\star}$ is not a zero of it. Therefore, to obtain a real
polynomial $Q(z)$ that has the same number of zeros on $\mathbb{R}$
as does $q_{\mu}(z)$, we can just compute the \noun{gcd} of the real
polynomials $r(z)$ and $s(z)$, where $q_{\mu}(z)=r(z)+is(z)$. This
alternative, which was already suggested in \citep{Conrad2016}, is
described in Algorithm \ref{AlgComplex1B}.

We highlight that we can also count the number of zeros of $p(z)$
in a given \emph{arc} of the unit circle\footnote{This works for all the cases considered in the advance, with few modifications
if necessary. For this reason we shall not comment about this possibility
further (we do remark, however, that the replacement $z\leftarrow\surd z$
cannot be employed anymore to this end, as this would lead to a wrong
result due to the fact that this map is not one-to-one). }. Let $\mathcal{J}=\left(\mathrm{e}^{i\alpha},\mathrm{e}^{i\beta}\right]$
be the referred arc of the unit circle. In the simplest case, we assume
that $0\leqslant\alpha<\beta\leqslant2\pi$, so that the interval
$\mathcal{J}$ is mapped to the interval $\mathcal{I}=(a,b]$ of the
real line, where $a=\omega(e^{i\alpha})$ and $b=\omega(e^{i\beta})$
(with the following conventions: $\lim_{\theta\rightarrow0}\omega(\mathrm{e}^{i\theta})=-\infty$
and $\lim_{\theta\rightarrow2\pi}\omega(\mathrm{e}^{i\theta})=\infty$).
The number of zeros of $p(z)$ on the arc $\mathcal{J}$ can thereby
be found by counting the number of real zeros that the polynomial
$Q(z)$ (defined by one of the two possible ways as described above),
has on the interval $\mathcal{I}$ of $\mathbb{R}$. In the case where
$\alpha>\beta$ (which corresponds to an interval on $\mathbb{S}$
that contains the point $z=1$), we need to split the algorithm into
two parts because, in this case, the interval $\mathcal{I}$ on $\mathbb{R}$
will be composed of two disjoint intervals --- namely, we have $\mathcal{I}(\alpha,\beta]=(-\infty,b]\cup(a,\infty)$.
Thus, the procedure $\textsc{rrc}\left[Q(z),a,b\right]$ must be replaced
by $\textsc{rrc}\left[Q(z),-\infty,b\right]+\textsc{rrc}\left[Q(z),a,\infty\right]$
in this case. Finally, if the point $z=1$ belongs to the interval
$\mathcal{J}$ and $p(1)=0$, then we should add $1$ to the final
result. This is described in Algorithm \ref{AlgComplexAB} (for the
sake of simplicity, we considered $Q(z)=q_{\mu}(z)q_{\mu}^{\star}(z^{\star})$
in the pseudo-code).

Moreover, it is clear that we can also locate and isolate the zeros
on the unit circle of a given polynomial through these algorithms.
In fact, after we map the zeros of the polynomial $p(z)$ on the unit
circle to the real line through the Cayley transformations (\ref{Cayley}),
we can find an interval $\mathcal{I}\subset\mathbb{R}$ containing
all the real zeros of the transformed polynomial $q_{\mu}(z)$; then,
from Sturm or Akritas procedures, we can refine this interval, for
example by a bisection method, so that we obtain a list $\left\{ \mathcal{I}_{1},\ldots,\mathcal{I}_{n}\right\} $
of intervals, each one containing exactly one real zero of $p(z)$
--- indeed, in many symbolic computing software, is this list of
isolating intervals that is returned by the implemented procedures
of Sturm and Akritas, see \citep{Akritas1989,AkritasVigklas2010}.
This list of isolating intervals of $q_{\mu}(z)$ provides a corresponding
list of arcs on the unit circle that isolate and locate the zeros
of $p(z)$ on $\mathbb{S}$. These intervals can be further refined
recursively so that approximated values for the zeros are returned.

Finally, notice that we can also count the number of zeros that a
polynomial has in any circle or straight line of the complex plane
by considering a suitable Möbius transformation $m(z)=\frac{az+b}{cz+d}$,
$ad-bc\neq0$, in place of Cayley transformation (\ref{Cayley}). 

\section{Real and self-conjugate polynomials }

The algorithms presented in the previous section apply to arbitrary
complex polynomials. In the most important cases, however, the coefficients
of the testing polynomial enjoy certain symmetries which allow us
to implement faster algorithms. Henceforward, we shall specialize
in classes of polynomials whose zeros are either symmetric with respect
to the real line or to the unit circle. We shall call a complex polynomial
whose zeros are all symmetric to $\mathbb{R}$ as a \emph{self-conjugate
polynomial} and a complex polynomial whose zeros are all symmetric
to $\mathbb{S}$ as a \emph{self-inversive polynomial}. 

\vspace{0.5cm}
\begin{algorithm} 
\caption{\textsc{Self-conjugate polynomials}} 
\label{AlgConjugate}
\SetKwInOut{Input}{input} \SetKwInOut{Output}{output}
\Input{ A  self-conjugate polynomial $p(z)$ of degree $n$.}
\Output{ The number of distinct zeros of $p(z)$ on the unit circle.}
$n\coloneqq \mathrm{degree}(p(z))$\;
$q(z)\coloneqq\left( z+i\right)^{n} p\left(\frac{ z-i}{ z+i}\right)$\;
\If{$q(z) \neq q^{\star}(z^{\star})$}
  {
  $q(z)\leftarrow  q\left( z\right)q^{\star}\left( z^{\star}\right)$\;
  $q(z)\leftarrow q(\surd z)$\;
  $N\coloneqq 2 \ \textsc{rrc}\left[q(z),0,\infty \right]$\;
    \lIf{$p(-1)=0$}{$N\leftarrow N+1$ \textbf{end}}
    \lIf{$p(1)=0$ }{$N\leftarrow N+1$ \textbf{end}}
  \Return{$N$};
  }
$N\coloneqq \textsc{rrc}\left[q(z),-\infty,\infty \right]$\;
\lIf{$p(1)=0$ }{$N\leftarrow N+1$ \textbf{end}}
\Return{$N$}.
\end{algorithm}
\vspace{0.5cm}

Let us consider in this section the analysis of self-conjugate polynomials
(self-inversive polynomials will be considered in the next section).
If $p(z)$ is self-conjugate, then, for any zero $\zeta$ of $p(z)$,
the complex conjugate number $\zeta^{\star}$ is also a zero of it.
Of course, any real polynomial has this property, but there can be
non-real polynomials with this property as well. The necessary and
sufficient condition for a complex polynomial $p(z)=p_{n}z^{n}+\cdots+p_{0}$
of degree $n$ to be self-conjugate is that $p_{n}\neq0$ and that
there exists a fixed complex number $\epsilon$ of modulus $1$ such
that, $p(z)=\epsilon p^{\star}(z^{\star})$ --- see \citep{Vieira2019}
for the proof. From this we can see that the coefficients of any self-conjugate
polynomial $p(z)$ of degree $n$ satisfy the properties $p_{k}=\epsilon p_{k}^{\star}$
for each ranging from $0$ to $n$. Real polynomials are those self-conjugate
polynomials with $\epsilon=1$. 

Notice that even when $p(z)$ is a real polynomial, the transformed
polynomial $q_{\mu}(z)$ is not necessarily real. Thus, to count the
number of zeros of a self-conjugate polynomial on the unit circle,
we need to compute, as an intermediary step, the real polynomial $Q(z)$
by one of the two methods discussed in Section \ref{Section-GeneralComplex}.
However, as we shall see in the following, the polynomials $Q(z)$
have additional symmetries when $p(z)$ is self-conjugate, which allow
us to improve the algorithms. 

Let us first consider that $Q(z)$ is defined as $Q(z)=q_{\mu}(z)q_{\mu}^{\star}(z^{\star}).$
In this case, the following theorem shows us that if $p(z)$ is self-conjugate,
then the polynomial $Q(z)$ has only even powers of $z$:
\begin{thm}
\label{TheoremSC}Let $p(z)$ be a self-conjugate polynomial of degree
$n$ such that $p(1)\neq0$. Then, the polynomial $Q(z)=q_{\mu}(z)q_{\mu}^{\star}(z^{\star})$
will be a real polynomial of degree $n$ in the variable $z^{2}$. 
\end{thm}
\begin{proof}
According to (\ref{q_mu}), we have that, \sloppy $Q(z)=q_{\mu}(z)q_{\mu}^{\star}(z^{\star})=(z^{2}+1)^{n}p(\mu(z))p^{\star}(\mu(z^{\star}))$,
which is clearly a real polynomial. If, moreover, $p(z)$ is self-conjugate,
then we get that $Q(z)=\epsilon^{-1}(z^{2}+1)^{n}p(\mu(z))p(\mu^{\star}(z^{\star})).$
But it follows from (\ref{Cayley}) that $\mu^{\star}(z^{\star})=1/\mu(z)=\mu(-z)$,
so that we obtain $Q(z)=\epsilon^{-1}(z^{2}+1)^{n}p(\mu(z))p\left(\mu(-z)\right)$.
Thus, we plainly see that $Q(-z)=Q(z)$, from which we conclude that
$Q(z)$ has only even powers of $z$.
\end{proof}
Hence, provided that $z=0$ is not a zero of $q_{\mu}(z)$ --- which
is the same of saying that $z=-1$ is not a zero of $p(z)$ ---,
the number of real zeros of $q_{\mu}(z)$ will be twice the number
of the positive zeros of $Q\left(\surd z\right)$, counted without
multiplicity. This property allows us to modify Algorithm \ref{AlgComplex1A}
by replacing $Q(z)$ with $Q\left(\surd z\right)$ and the procedure
$\textsc{rrc}[Q(z),-\infty,\infty]$ with $2\,\textsc{rrc}[Q\left(\surd z\right),0,\infty]$,
so that a faster algorithm for self-conjugate polynomials is achieved
(because now $Q\left(\surd z\right)$ has degree $n$ instead of $2n$).
Notice, however, that the eventual zero of $p(z)$ at $z=-1$ should
be counted separately, in the same fashion as the eventual zero of
$p(z)$ at the point $z=1$. This is exemplified in Algorithm \ref{AlgConjugate}.
We should remark, however, that this algorithm is not suitable for
counting the number of zeros that a self-conjugate polynomial $p(z)$
of degree $n$ has in a finite interval $\mathcal{J}=(\mathrm{e}^{i\alpha},\mathrm{e}^{i\beta}]$
of $\mathbb{S}$ because the change of variable $z\leftarrow\surd z$
is not a one-to-one map. In fact, in this case we can no longer guarantee
that the number of zeros that $Q(z)$ has on this interval corresponds
to the twice the number of zeros of $Q(\surd z)$ in the respective
positive interval of the real line. 

\vspace{0.5cm}
\begin{algorithm} 
\setcounter{algocf}{3}
\caption{ \textsc{Self-adjoint polynomials}} 
\label{AlgAdjoint}
\SetKwInOut{Input}{input} \SetKwInOut{Output}{output}
\Input{ A self-adjoint polynomial $p(z)$ of degree $n$.}
\Output{ The number of distinct zeros of $p(z)$ on the unit circle.}

$n\coloneqq \mathrm{degree}(p(z))$\;
$q(z)\coloneqq\left(z+i\right)^{n} p\left(\frac{z-i}{z+i}\right)$\;
$N\coloneqq \textsc{rrc}\left[q(z),-\infty,\infty \right]$\;
\lIf{$p(1)=0$}{$N \leftarrow N+1$ \textbf{end}}
\Return{$N$}.
\end{algorithm}
\vspace{0.5cm}

The another possibility is to define $Q(z)$ through \sloppy $Q(z)=\textsc{gcd}\left[\frac{1}{2}\left(q_{\mu}(z)+q_{\mu}^{\star}(z^{\star})\right),\frac{1}{2i}\left(q_{\mu}(z)-q_{\mu}^{\star}(z^{\star})\right)\right]$.
This has the advantage of providing a real polynomial $Q(z)$ whose
degree is at most $n$. In fact, we have the following:
\begin{thm}
Let $p(z)$ be a self-conjugate polynomial of degree $n$. Then, the
degree of the polynomial $Q(z)=\textsc{gcd}\left[\frac{1}{2}\left(q_{\mu}(z)+q_{\mu}^{\star}(z^{\star})\right),\frac{1}{2i}\left(q_{\mu}(z)-q_{\mu}^{\star}(z^{\star})\right)\right]$
will match the number of zeros of $p(z)$ that are symmetric to $\mathbb{S}$. 
\end{thm}
\begin{proof}
We have seen in Theorems \ref{TheoremNumberRoots} and \ref{TheoremNumberSymmetric}
that any pair of zeros of $p(z)$ that are on, or are symmetric to,
the unit circle are mapped into a pair of real, or non-real complex
conjugate, zeros of $q_{\mu}(z)$. On the other hand, Theorem \ref{TheoremZerosGCD}
states that the zeros of the polynomial $Q(z)$ as defined above are
precisely the complex conjugate zeros of $q_{\mu}(z)$. These two
assertions imply that the degree of $Q(z)$ equals the number of zeros
of $p(z)$ that are symmetric to $\mathbb{S}$.
\end{proof}
The corresponding algorithm is the same as Algorithm \ref{AlgComplex1B}
and does not need to be presented again. We highlight, nevertheless,
that if all the zeros of a self-conjugate polynomial $p(z)$ which
not lie on $\mathbb{S}$ are not symmetric to $\mathbb{S}$ either,
then the degree of $Q(z)$ provides directly the number of zeros of
$p(z)$ on $\mathbb{S}$, so that in this case there is no need of
using any \noun{rrc} method whatsoever. 

\section{Self-inversive, self-adjoint and skew-adjoint polynomials}

In this section we shall consider the case of a complex polynomial
$p(z)$ whose zeros are all symmetric with respect to the unit circle.
This means that, for any zero $\zeta$ of $p(z)$, the complex number
$1/\zeta^{\star}$ is also a zero of it. Any polynomial of this kind
is called a \emph{self-inversive polynomial} and the necessary and
sufficient condition for a polynomial $p(z)=p_{n}z^{n}+\cdots+p_{0}$
of degree $n$ to be self-inversive is that $p_{n}p_{0}\neq0$ and
that there exists a complex number $\epsilon$ with modulus $1$ such
that $p(z)=\epsilon z^{n}p^{\star}\left(1/z^{\star}\right)$ ---
see \citep{Vieira2019} for the proof. The coefficients of any self-inversive
polynomial $p(z)$ of degree $n$ satisfy the properties $p_{n-k}=\epsilon p_{k}^{\star}$,
for $k$ ranging from $0$ to $n$. If a given polynomial is self-inversive
with $\epsilon=1$ (resp. $\epsilon=-1$) we shall call it a \emph{self-adjoint
}(\emph{skew-adjoint}) \emph{polynomial}. 

\vspace{0.5cm}
\begin{algorithm} 
\caption{\textsc{Self-inversive polynomials}} 
\label{Alg-inv}
\SetKwInOut{Input}{input} \SetKwInOut{Output}{output}
\Input{ A self-inversive polynomial $p(z)$ of degree $n$.}
\Output{ The number of distinct zeros of $p(z)$ on the unit circle.}
$n\coloneqq \mathrm{degree}(p(z))$\;
$\epsilon \coloneqq p_n/p_0^{\star}$\;
\lIf{$\epsilon\neq1$}{$p(z)\leftarrow p\left(z \epsilon^{-1/n}\right)$ \textbf{end}}
$q(z)\coloneqq\left(z+i\right)^{n} p\left(\frac{z-i}{z+i}\right)$\;
$N\coloneqq\textsc{rrc}\left[q(z),-\infty,\infty \right]$\;
\lIf{$p(1)=0$}{$N \leftarrow N+1$ \textbf{end}}
\Return{$N$.}
\end{algorithm}
\vspace{0.5cm}

The following theorem shows that there is a one-to-one correspondence
between the sets of self-inversive and self-conjugate polynomials,
as well as between the sets of self-adjoint and real polynomials.
\begin{thm}
\label{Theorem-SI-SC}Let $p(z)$ be a self-inversive polynomial.
Then, the transformed polynomial $q_{\mu}(z)$ defined in (\ref{q_mu})
will be a self-conjugate polynomial. Moreover, if $p(z)$ is a self-adjoint
(resp. skew-adjoint) polynomial, then the transformed polynomial $q_{\mu}(z)$
will be a real (imaginary) polynomial. Similarly, let $p(z)$ be a
self-conjugate polynomial. Then the polynomial $q_{\omega}(z)$ defined
in (\ref{q_mu}) will be a self-inversive polynomial and if $p(z)$
is a real (imaginary) polynomial, then $q_{\omega}(z)$ will be a
self-adjoint (skew-adjoint) polynomial.
\end{thm}
\begin{proof}
The result follows from Theorems \ref{TheoremNumberRoots} and \ref{TheoremNumberSymmetric}.
Explicitly, we have the following: Let us first suppose $p(z)$ self-inversive.
Then, $q_{\mu}(z)=\left(z+i\right)^{n}p(\mu(z))=\epsilon\left(z+i\right)^{n}\mu(z)^{n}p^{\star}\left(1/\mu^{\star}\left(z\right)\right)=\epsilon\left(z-i\right)^{n}p^{\star}\left(1/\mu^{\star}\left(z\right)\right).$
But we have the identity $1/\mu^{\star}(z)=\mu(z^{\star})$, from
which $q_{\mu}(z)$ simplifies to $q_{\mu}(z)=\epsilon\left(z-i\right)^{n}p^{\star}\left(\mu\left(z^{\star}\right)\right)=\epsilon q_{\mu}^{\star}(z^{\star})$;
this proves that $q_{\mu}(z)$ is self-conjugate. Besides, notice
that if $\epsilon=1$ (resp. $\epsilon=-1$), so that $p(z)$ is self-adjoint
(skew-adjoint) polynomial, then we get that $q_{\mu}(z)$ will be
a real (imaginary) polynomial because the value of $\epsilon$ is
preserved during this transformation. Now, suppose $p(z)$ a self-conjugate
polynomial. Then, we get that $q_{\omega}(z)=\left(\tfrac{i}{2}\right)^{n}\left(z-1\right)^{n}p(\omega(z))=\epsilon\left(z-1\right)^{n}p^{\star}(\omega^{\star}(z)).$
But we have the identity $\omega^{\star}\left(z\right)=\omega(1/z^{\star})$,
from which we obtain $q_{\omega}(z)=\epsilon\left(\tfrac{i}{2}\right)^{n}\left(z-1\right)^{n}p^{\star}\left(\omega\left(1/z^{\star}\right)\right)=\epsilon z^{n}q_{\omega}^{\star}\left(1/z^{\star}\right)$;
this proves that $p(z)$ is self-inversive. Moreover, if $\epsilon=1$
(resp. $\epsilon=-1$), so that $p(z)$ is a real (imaginary) polynomial,
then we see that $q_{\omega}(z)$ becomes a self-adjoint (skew-adjoint)
polynomial. 
\end{proof}
Now, let us see how we can count the number of zeros that a self-adjoint
or a self-inversive polynomial has on the unit circle. Let us first
consider the case where $p(z)$ is self-adjoint. In this case, Theorem
\ref{Theorem-SI-SC} ensures that $q_{\mu}(z)$ is already a real
polynomial, so that there is no need of computing the polynomials
$Q(z)$ introduced in Section \ref{Section-GeneralComplex}. This
results in Algorithm \ref{AlgAdjoint}, which is usually faster than
Algorithms \ref{AlgComplex1A} and \ref{AlgComplex1B}, as the steps
concerning the computation of $Q(z)$ are absent. Of course, this
algorithm also works for skew-adjoint polynomials: we just have to
make the additional replacement $q_{\mu}(z)\leftarrow iq_{\mu}(z)$. 

Let us now suppose $p(z)$ self-inversive with $\epsilon\neq1$. In
this case, the transformed polynomial $q_{\mu}(z)$ is not real anymore.
A direct approach to work around this issue would be to compute the
real polynomial $Q(z)$ as introduced in Section \ref{Section-GeneralComplex},
but we actually have a better alternative: as the next theorem shows,
a self-inversive polynomial $p(z)$ can always be transformed into
a self-adjoint polynomial $s(z)$, whose degree is the same as that
of $p(z)$, through a simple change of variable:
\begin{thm}
\label{Theorem-Rotation}Let $p(z)$ be a self-inversive polynomial
of degree $n$ such that  $\epsilon\neq1$. Then, there exist $n$
values for the real variable $\phi$ in the interval $0<\phi\leqslant2\pi$
for which the composition $s(z)=p(\mathrm{e}^{i\phi}z)$ will provide
a self-adjoint polynomial of degree $n$. The possible values of $\phi$
are related with $\epsilon$ through the formula $\phi=\left(i\log\epsilon-2\pi k\right)/n$,
for $1\leqslant k\leqslant n$, such that $\epsilon=\mathrm{e}^{-in\phi}$
for any admissible value of $\phi$. Conversely, if $s(z)$ is a self-adjoint
polynomial of degree $n$, then $p(z)=s\left(z\mathrm{e}^{-i\phi}\right)$
will provide a self-inversive polynomial of degree $n$ such that
$\epsilon=\mathrm{e}^{in\phi}$.
\end{thm}
\begin{proof}
Let $p(z)=p_{n}z^{n}+p_{n-1}z^{n-1}+\cdots+p_{1}z+p_{0}$ be a self-inversive
polynomial of degree $n$. Making the change of variable $z\leftarrow\mathrm{e}^{i\phi}z$,
we shall get another polynomial of degree $n$, $s(z)=s_{n}z^{n}+s_{n-1}z^{n-1}+\cdots+s_{1}z+s_{0}$,
where $s_{k}=p_{k}\mathrm{e}^{ik\phi}$, $0\leqslant k\leqslant n$.
Now, $p(z)$ is self-inversive so that its coefficients satisfy the
relations $p_{n-k}=\epsilon p_{k}^{\star}$, $0\leqslant k\leqslant n$.
From this we can see that the condition for $s(z)$ to be a self-adjoint
polynomial is that $\epsilon=\mathrm{e}^{-in\phi}.$ Inverting this
relation, we conclude that $\phi$ can assume $n$ distinct values
in the interval $0<\phi\leqslant2\pi$, which are determined by the
formula $\phi_{k}=\left(i\log\epsilon-2\pi k\right)/n$, for $0\leqslant k\leqslant n$;
each one of them leads to a particular self-adjoint polynomial $s^{(k)}(z)=p\left(z\mathrm{e}^{i\phi_{k}}\right).$
Notice that, in terms of $\epsilon$, we can also write: $s^{(k)}(z)=p\left(z\epsilon^{-1/n}/\varrho_{n}^{k}\right)$,
$1\leqslant k\leqslant n$, where $\varrho_{n}^{k}=\mathrm{e}^{2\pi i\frac{k}{n}}$
denotes the $k$th root of unity of degree $n$. Finally, given a
self-adjoint polynomial $s(z)$ of degree $n$, then it is clear that
$p(z)=s(\mathrm{e}^{-i\phi}z)$ will be a self-inversive polynomial
with $\epsilon=\mathrm{e}^{in\phi}$ for any admissible value $\phi_{k}$
of $\phi$, as above. 
\end{proof}
We highlight that Theorem \ref{Theorem-Rotation} means that any self-inversive
polynomial can be thought of a rotated self-adjoint polynomial. In
fact, if $\zeta_{1},\ldots,\zeta_{n}$ denote the zeros of a self-inversive
polynomial $p(z)$ of degree $n$, and $\sigma_{1}^{(j)},\ldots,\sigma_{n}^{(j)}$
the correspondent zeros of the self-adjoint polynomials $s^{(j)}=p(\mathrm{e}^{i\phi_{j}})$,
$1\leqslant j\leqslant n$, as provided by Theorem \ref{Theorem-Rotation},
then it is a straightforward matter to show that $\sigma_{k}^{(j)}=\mathrm{e}^{-i\phi_{j}}\zeta_{k}=\epsilon^{-1/n}\zeta_{k}/\varrho_{n}^{j}$,
for any $j$ and $k$ running from $1$ to $n$, where $\varrho_{n}^{j}=\mathrm{e}^{2\pi i\frac{j}{n}}$.
Therefore, we can say that the zeros of $s^{(j)}(z)$ are\emph{ rotated}
with respect to the zeros of $p(z)$ by an angle equal to $\epsilon^{-1/n}/\varrho_{n}^{j}$
in the clockwise direction. Theorem \ref{Theorem-Rotation} also shows
us that if we rotate the zeros of a given self-inversive polynomial
$p(z)$ of degree $n$ by an angle equal to any root of unity of degree
$n$, then we shall obtain another self-inversive polynomial with
the same $\epsilon.$ This means there are exactly $n$ self-inversive
polynomials conjugated in this way for the same value of $\epsilon$. 

Now, Theorem \ref{Theorem-Rotation} enables us to implement a specific
algorithm for counting the number of zeros that a self-inversive polynomial
has on $\mathbb{S}$. This is described in Algorithm \ref{Alg-inv}.
Because it differs from Algorithm \ref{AlgAdjoint} only by the replacement
$p(z)\leftarrow p\left(z\mathrm{e}^{-1/n}\right)$ (we can choose
$\varrho_{n}^{k}=1$), both algorithms have essentially the same complexity.

\section{Self-reciprocal and skew-reciprocal polynomials. An application to
Salem polynomials}

As the last case to be discussed in this work, let us suppose the
possibility of a complex polynomial $p(z)$ of degree $n$ which is,
at the same time, self-conjugate and self-inversive. From the properties
$p(z)=\epsilon p^{\star}\left(z^{\star}\right)$ and $p(z)=z^{n}\epsilon p^{\star}\left(1/z^{\star}\right)$
which should be satisfied by self-conjugate and self-inversive polynomials,
respectively, it follows that such polynomials should satisfy the
property: $p(z)=\epsilon z^{n}p\left(1/z\right).$ Contrary to the
previous cases, however, $\epsilon$ can assume only the values $1$
or $-1$ (to see why, replace $z$ by $1/z$ in the formula above),
which means that $p(z)$ must be a real polynomial. In the first case
$\left(\epsilon=1\right)$ we say that $p(z)$ is a \emph{self-reciprocal
polynomial}, while in the second case $\left(\epsilon=-1\right)$,
$p(z)$ is often called a \emph{skew-reciprocal polynomial}. The coefficients
of any self-reciprocal and any skew-reciprocal polynomial satisfy,
respectively, the relations $p_{n-k}=p_{k}$ and $p_{n-k}=-p_{k}$,
for any $k$ from $0$ to $n$, see \citep{Vieira2019}. 

\vspace{0.5cm}
\begin{algorithm}
\caption{\textsc{Self-reciprocal or skew-reciprocal  polynomials}}
\label{AlgReciprocal}
\SetKwInOut{Input}{input} \SetKwInOut{Output}{output}
\Input{
A self-reciprocal or skew-reciprocal polynomial $p(z)$ of degree $n$.}
\Output{The number of  distinct zeros of $p(z)$ on the unit circle.}
$N \coloneqq 0$\;
\lIf{$p( 1)=0$}{$N \leftarrow N+1$ \textbf{end}}
\lIf{$p(-1)=0$}{$N \leftarrow N+1$ \textbf{end}} 
\lWhile{$p( 1)=0$}{$p(z) \leftarrow \frac{p(z)}{z-1}$ \textbf{end}}
\lWhile{$p(-1)=0$}{$p(z) \leftarrow \frac{p(z)}{z+1}$ \textbf{end}}
$n \coloneqq \mathrm{degree}(p(z))$\;
$q(z) \coloneqq \left( \surd z + i \right)^{n} p \left(\frac{\surd z - i}{\surd z + i}\right)$\;
$N \leftarrow N + \textsc{rrc}\left[q(z),-\infty,\infty \right]$\;
\Return{$N$.}
\end{algorithm}
\vspace{0.5cm}

The behaviour of self-reciprocal and skew-reciprocal polynomials at
$z=\pm1$ is described in the following 
\begin{thm}
\label{TheoremZerosSelfSkew}If $z=1$ is a zero of a self-reciprocal
polynomial $p(z)$, then its multiplicity is even. Moreover, if $z=-1$
is a zero of a self-reciprocal polynomial $p(z)$ of degree $n$,
then its multiplicity is the same as the parity of $n$. Similarly,
if $z=1$ is a zero of a skew-reciprocal polynomial $p(z)$, then
its multiplicity is odd. Moreover, if $z=-1$ is a zero of a skew-reciprocal
polynomial $p(z)$ of degree $n$, then its multiplicity is the opposite
of the parity of $n$.
\end{thm}
\begin{proof}
First of all, notice that any skew-reciprocal polynomial $p(z)$ has
a zero at $z=1$, as well as, any self-reciprocal (resp. skew-reciprocal)
polynomial of odd (even) degree has a zero at $z=-1$. These propositions
follow directly from the evaluation of $p(z)$ at $z=\pm1$ and from
the symmetry of its coefficients. Now, suppose that $z=1$ is a zero
of a self-reciprocal polynomial $p(z)$ of odd multiplicity, say $m=2k+1$,
$k\in\mathbb{N}$. Then, $p(z)=\left(z-1\right)^{2k+1}P(z)$ and,
as $\left(z-1\right)^{2k+1}$ is a skew-reciprocal polynomial, it
follows that $P(z)$ would be a skew-reciprocal polynomial without
zeros at $z=1$, a contradiction. Similarly, suppose that $z=-1$
is a zero of $p(z)$ with multiplicity $m$. Then $p(z)=\left(z+1\right)^{m}P(z)$
and, as $\left(z+1\right)^{m}$ is self-reciprocal, it follows that
$P(z)$ is self-reciprocal without zeros at $z=-1$. Thus $P(z)$
must have even degree, which implies that $m$ has the same parity
as the degree of $p(z)$. Moreover, suppose that $z=1$ is a zero
of a skew-reciprocal polynomial $p(z)$ of even parity, say, $m=2k$,
$k\in\mathbb{N}$. Then, $p(z)=\left(z-1\right)^{2k}P(z)$, and as
$\left(z-1\right)^{2k}$ is self-reciprocal, it follows that $P(z)$
would be a skew-reciprocal polynomial without zeros at $z=1$, again
a contradiction. Similarly, suppose that $z=-1$ is a zero of $p(z)$
with multiplicity $m$. Then $p(z)=\left(z+1\right)^{m}P(z)$ and,
as $\left(z+1\right)^{m}$ is self-reciprocal, it follows that $P(z)$
is skew-reciprocal without zeros at $z=-1$. Thus $P(z)$ must have
odd degree, which implies that the parity of $m$ is opposed to the
parity of the degree of $p(z)$. 
\end{proof}
Besides, the action of the Cayley transformation (\ref{Cayley}) over
a self-reciprocal or skew-reciprocal polynomial $p(z)$ of degree
$n$ and with no zeros at $z=\pm1$ leads to a real polynomial $q_{\mu}(z)$
of degree $n$ in the variable $z^{2}$, as it is shown in the next
theorem:
\begin{thm}
\label{Theorem-Reciprocal}Let $p(z)$ be a self-reciprocal polynomial
of even degree, say $n=2m$, with no zeros at $z=\pm1$. Then, the
Cayley-transformed polynomial $q_{\mu}(z)$ will be a real polynomial
of degree $m$ in the variable $z^{2}$. 
\end{thm}
\begin{proof}
Let $p(z)$ be a self-reciprocal polynomial of even degree, say, $n=2m$.
Because the coefficients of any self-reciprocal polynomial satisfy
the relations $p_{n-k}=p_{k}$, $0\leqslant k\leqslant n$, it follows
that $p(z)$ can be written as, $p(z)=p_{m}z^{m}+\sum_{k=0}^{m-1}p_{k}\left(z^{2m-k}+z^{k}\right)=z^{m}\left[p_{m}+\sum_{k=0}^{m-1}p_{k}\left(z^{m-k}+z^{k-m}\right)\right]$.
On the other hand, the transformed polynomial $q_{\mu}(z)$ defined
in (\ref{q_mu}) becomes, $q_{\mu}(z)=\left(z^{2}+1\right)^{m}p_{m}+\sum_{k=0}^{m-1}p_{k}\left(z^{2}+1\right)^{k}\left[\left(z+i\right)^{2m-2k}+\left(z-i\right)^{2m-2k}\right]$,
after a simplification. Therefore, we can plainly see that $q_{\mu}(z)$
is an even function of $z$, which means that $q(z)$ is in fact a
polynomial of degree $m$ on the variable $z^{2}$. Furthermore, $q_{\mu}(z)$
is also a real polynomial because all the imaginary terms inside the
brackets will cancel after we expand the binomials.
\end{proof}
We can also show through similar arguments that, if $p(z)$ is a skew-reciprocal
polynomial with no zero at $z=-1$, then the transformed polynomial
$q_{\mu}(z)$ will be a pure imaginary polynomial in the variable
$z^{2}$. Notice as well that any zero of $p(z)$ at $z=1$ is mapped
to infinity, as we have seen, so that the degree of $q_{\mu}(z)$
will be less than the degree of $p(z)$ in this case; similarly, any
zero of $p(z)$ at $z=-1$ is mapped to zero, so that $p(z)$ will
be multiplied by some power of $z$ in this case.

Theorems \ref{TheoremZerosSelfSkew} and \ref{Theorem-Reciprocal}
provide a great improvement to the algorithms presented in the previous
sections. In fact, to count the number of zeros that a self-reciprocal
or skew-reciprocal polynomial $p(z)$ has on $\mathbb{S}$, we can
first remove its zeros (if any) at the points $z=1$ and $z=-1$ by
dividing it successively by $z-1$ and $z+1$ so that a self-reciprocal
polynomial $r(z)$, with no zeros at $z=\pm1$ is obtained in place.
Then, we can compute the Cayley-transformation of $r(z)$ and, thanks
to Theorem \ref{Theorem-Reciprocal}, to make the replacement $z\leftarrow\surd z$,
which provides a polynomial of the half of the degree of $r(z)$.
The number of real zeros of $r(z)$ will, therefore, equal the number
of zeros on $\mathbb{S}$ of $p(z)$, excepting its possible zeros
at $z=\pm1$, which can be verified separately. This is described
in Algorithm \ref{AlgReciprocal} . 

Finally, notice that from Algorithm \ref{AlgReciprocal} we can easily
test if a given polynomial is (or not is) a \emph{Salem polynomial}
without knowing explicitly its zeros. A Salem polynomial $p(z)$ is
a monic, irreducible, self-reciprocal polynomial of degree $n\geqslant4$
with integer coefficients, whose all but two zeros lie on the unit
circle. Their two zeros not lying on $\mathbb{S}$ are necessarily
real, positive and the reciprocal of each other. The greatest real
zero of a Salem polynomial is called its \emph{Salem number}, and
if the value of this number is less than the \emph{smallest Pisot
number} $\varrho\approx1.3247179$ (which corresponds to the unique
real zero of the Pisot polynomial $p(z)=z^{3}-z-1$), then it is usually
called a \emph{small} \emph{Salem number}\footnote{Actually, a Salem number is often called small if it is less than
$1.3$, see \citep{Boyd1977,Mossinghoff1998}. We think, however,
that our definition is more appropriate, as the value $1.3$ is quite
arbitrary, in contrast with the smallest Pisot number $\varrho$ which
naturally plays an important role in the field. }\emph{. }A \emph{Pisot polynomial} is a monic, irreducible, non-self-reciprocal
integer polynomial that has only one zero outside the unit circle,
which is its Pisot number. Up to date there were found only 47 small
Salem numbers\footnote{Please notice, however, that with the alternative definition adopted
here, the list of known small Salem numbers presented for instance
in \citep{Boyd1977,Mossinghoff1998} would be increased by some few
new entries.}, the smallest one being the \emph{Lehmer number} $\lambda\approx1.1762808$,
the greatest real zero of the polynomial $L(z)=z^{10}+z^{9}-z^{7}-z^{6}-z^{5}-z^{4}-z^{3}+z+1$
--- see \citep{Boyd1977,Mossinghoff1998}. It is still an open problem
to know if Lehmer's number $\lambda$ is the smallest Salem number,
or even if there exists a smallest Salem number after all. We highlight
that Algorithm \ref{AlgReciprocal} provides a powerful tool to look
for polynomials with small Salem numbers and, in a more general way,
to search for polynomials with small Mahler measure. We in fact succeeded
in reproducing all small Salem numbers known up to date with an improved
form of Algorithm \ref{AlgReciprocal} running in a simple desktop
computer. We intend to report a detailed analysis of these researches
in the future.

\section*{Acknowledgements}

The author thanks the Editor and Referees for their valuable suggestions.
This work was supported by Coordination for the Improvement of Higher
Education Personnel (CAPES). 

This work is licensed under a CC-BY-NC-ND License. The final version
is published in the Journal of Computational and Applied Mathematics,\emph{
\url{https://doi.org/10.1016/j.cam.2020.113169}.}

\bibliographystyle{elsarticle-num}
\bibliography{RootsU}

\end{document}